\documentclass[12pt,reqno]{amsart}
\usepackage{enumerate, latexsym, amsmath, amsfonts, amssymb, amsthm, graphicx, color}
\def\pmod #1{\ ({\rm{mod}}\ #1)}
\def\Z{\Bbb Z}

\def\Q{\Bbb Q}

\def\bg{\bigg}
\def\({\bg(}
\def\){\bg)}

\def\det{{\rm det}}
\def\sgn{{\rm sgn}}

\def\sgn{{\rm sgn}}

\def\Gal{{\rm Gal}}

\def\ve{\varepsilon}

\def\Ack{\medskip\noindent {\bf Acknowledgments}}
\theoremstyle{plain}
\newtheorem{theorem}{Theorem}

\newtheorem{lemma}{Lemma}
\newtheorem{corollary}{Corollary}

\theoremstyle{definition}

\theoremstyle{remark}

\makeatletter
\@namedef{subjclassname@2010}{%
  \textup{2010} Mathematics Subject Classification}
\makeatother
 \vspace{4mm}

\begin{document}
 \baselineskip=17pt
\hbox{} {}
\medskip
\title[On some determinants involving cyclotomic units]
{On some determinants involving cyclotomic units}
\date{}
\author[Hai-Liang Wu] {Hai-Liang Wu}

\thanks{
\newline \indent Supported by the National Natural Science
Foundation of China (Grant No. 11571162).}

\address {(Hai-Liang Wu) Department of Mathematics, Nanjing
University, Nanjing 210093, People's Republic of China}
\email{{\tt whl.math@smail.nju.edu.cn}}

\begin{abstract}
For each odd prime $p$, let $\zeta_p$ denote a primitive $p$-th root of unity. In this paper, we study the determinants of some matrices
with cyclotomic unit entries. For instance, we show that when $p\equiv 3\pmod4$ and $p>3$ the determinant of the matrix
$\(\frac{1-\zeta_p^{j^2k^2}}{1-\zeta_p^{j^2}}\)_{1\le j,k\le (p-1)/2}$ can be written as $(-1)^{\frac{h(-p)+1}{2}}(a_p+b_pi\sqrt{p})$ with $a_p,b_p\in\frac12\Z$ and
$$\begin{cases}\nu_p(a_p)=\nu_p(b_p)=\frac{p-3}{8}&\mbox{if}\ p\equiv 3\pmod8,
\\\nu_p(a_p)=\nu_p(b_p)+1=\frac{p+1}{8}&\mbox{if}\ p\equiv 7\pmod8,\end{cases}$$
where $\nu_p(x)$ denotes the $p$-adic order of a $p$-adic integer $x$, and $h(-p)$ denotes the class number of the field $\Q(\sqrt{-p})$.
Meanwhile, let $(\frac{\cdot}{p})$ denote the Legendre symbol. We have
$$2^{\frac{p+1}{2}}a_pb_p=(-1){^\frac{p+1}{4}}p^{\frac{p-3}{4}}\det [S(p)],$$
and
$$2^{\frac{p-1}{2}}(a_p^2-pb_p^2)=\frac{p-1}{2}(-p)^{\frac{p-3}{4}}\det [S(p)],$$
where $\det [S(p)]$ is the determinant of the $\frac{p-1}{2}$ by $\frac{p-1}{2}$ matrix $S(p)$ with entries $S(p)_{j,k}=(\frac{j^2+k^2}{p})$ for any $1\le j,k\le (p-1)/2$.
\end{abstract}

\maketitle

\section{Introduction}
\setcounter{lemma}{0}
\setcounter{theorem}{0}
\setcounter{corollary}{0}
\setcounter{remark}{0}
\setcounter{equation}{0}
\setcounter{conjecture}{0}
\setcounter{proposition}{0}

Let $p$ be an odd prime and let $\zeta_p$ be a primitive $p$-th root of unity. Let $\Z[\zeta_p]^{\times}$ denote the group of units
of $\Q(\zeta_p)$ and let $V_p$ be the group generated by the set
$$\{\pm \zeta_p, 1-\zeta_p^k: 1\le k\le p-1\}.$$ Then the group of cyclotomic units $U_p$ is defined by
$$U_p=V_p\cap\Z[\zeta_p]^{\times}.$$
The group of cyclotomic units has close relations with the class number of the field $\Q(\zeta_p+\zeta_p^{-1})$. Readers may see
\cite{Washington} for more details. On the other hand, the product
$$\prod_{1\le k\le p-1}(1-\zeta_p^{k^2})$$
has deep connections with the
class number and fundamental unit of the unique quadratic subfield of $\Q(\zeta_p)$. Indeed, as a corollary of the analytic class number
formula (cf. \cite[p.344 Theorem 2]{BS}), when $p\equiv1\pmod 4$ we have
\begin{equation}\label{equation A}
\prod_{1\le k\le p-1}(1-\zeta_p^{k^2})=\ve_p^{-h(p)}\sqrt{p},
\end{equation}
where $h(p)$ and $\ve_p>1$ are the class number and the fundamental unit of the field $\Q(\sqrt{p})$ respectively.
And when $p\equiv 3\pmod 4$ and $p>3$ we also have
\begin{equation}\label{equation B}
\prod_{1\le k\le p-1}(1-\zeta_p^{k^2})=(-1)^{\frac{h(-p)+1}{2}}i\sqrt{p},
\end{equation}
where $h(-p)$ denotes the class number of the field $\Q(\sqrt{-p})$. 

Recently Sun \cite{S19} investigated the determinants of matrices with Legendre symbol entries. He first studied the matrices
$$S(p)=\(\(\frac{j^2+k^2}{p}\)\)_{1\le j,k\le (p-1)/2},$$
$$T(\triangle,p)=\(\(\frac{j^2+\triangle k^2}{p}\)\)_{0\le j,k\le (p-1)/2},$$
and
$$S(\triangle,p)=\(\(\frac{j^2+\triangle k^2}{p}\)\)_{1\le j,k\le (p-1)/2},$$
where $(\frac{\cdot}{p})$ denotes the Legendre symbol and $\triangle$ is a quadratic non-residue modulo $p$. In this line, later Sun \cite{S19B} studied the determinants of some matrices concerning the tangent function.

Inspired by the above results, in this paper we mainly focus on the determinant of the matrix with cyclotomic unit entries. Let
$$C_p=\(\frac{1-\zeta_p^{j^2k^2}}{1-\zeta_p^{j^2}}\)_{1\le j,k\le (p-1)/2}.$$
We will see later that the determinant $\det[C_p]$ has close relations with determinants of the matrices
$S(p)$, $T(\triangle,p)$ and $S(\triangle,p)$. 

Now we are in the position to state our main results of this paper.

\begin{theorem}\label{Thm. p=3mod4}
Let $p\equiv 3\mod 4$ be a prime with $p>3$. Then we may write
$$\det[C_p]=(-1)^{\frac{h(-p)+1}{2}}(a_p+b_pi\sqrt{p}),$$
with $a_p,b_p\in\frac12\Z$ and
$$\begin{cases}\nu_p(a_p)=\nu_p(b_p)=\frac{p-3}{8}&\mbox{if}\ p\equiv 3\pmod8,
\\\nu_p(a_p)=\nu_p(b_p)+1=\frac{p+1}{8}&\mbox{if}\ p\equiv 7\pmod8,\end{cases}$$
where $\nu_p(x)$ denotes the $p$-adic order of a $p$-adic integer $x$.
Meanwhile,
$$2^{\frac{p+1}{2}}a_pb_p=(-1){^\frac{p+1}{4}}p^{\frac{p-3}{4}}\det [S(p)],$$
and
$$2^{\frac{p-1}{2}}(a_p^2-pb_p^2)=\frac{p-1}{2}(-p)^{\frac{p-3}{4}}\det [S(p)],$$
\end{theorem}

From Theorem \ref{Thm. p=3mod4}, it is easy to obtain the following result.
\begin{corollary}
Let notations be as in the Theorem \ref{Thm. p=3mod4}. Then
$$\nu_2(\det[S(p)])\ge \frac{p-3}{2}.$$
\end{corollary}

When $p\equiv 1\pmod 4$, there are $a,b\in\Z$ with $2\nmid a$ such that $p=a^2+b^2$. Then we have the following result.
\begin{theorem}\label{Thm. p=1mod4}
Let $p\equiv1\pmod4$ be a prime. Then we may write
$$\det[C_p]=\frac{\ve_p^{h(p)}}{\sqrt{p}}(\alpha_p+\beta_p\sqrt{p})\sqrt{\(\frac{2}{p}\)2p+2a\sqrt{p}},$$
with $\alpha_p,\beta_p\in\Q$. Meanwhile, we have
$$\pm 2^{\frac{p+1}{4}}b(\alpha^2_p-p\beta^2_p)=p^{\frac{p-1}{4}}\det[T(\triangle,p)].$$
\end{theorem}

\maketitle
\section{Proofs of Theorem 1.1--1.2}
\setcounter{lemma}{0}
\setcounter{theorem}{0}
\setcounter{corollary}{0}
\setcounter{remark}{0}
\setcounter{equation}{0}
\setcounter{conjecture}{0}

Throughout this section, we write $p=2m+1$.

For each $a\in\Z$ with $\gcd (a,p)=1$, clearly

$$(a\cdot1)^2\pmod p,\ \cdots,\ (a\cdot m^2)\pmod p$$
is a permutation on
$$1^2\pmod p,\ \cdots,\ m^2\pmod p.$$
We denote this permutation by $\tau_a$. Then we have the following result concerning the sign of $\tau_a$.

\begin{lemma}\label{Lem. permutation}
$$\sgn(\tau_a)=\begin{cases}1&\mbox{if}\ p\equiv 3\pmod4,
\\(\frac{a}{p})&\mbox{if}\ p\equiv 1\pmod4.\end{cases}$$
\end{lemma}

\begin{proof}
By definition our result follows from
$$\sgn(\tau_a)\equiv \prod_{1\le i<j\le m}\frac{a^2j^2-a^2i^2}{j^2-i^2}=a^{m(m-1)}\equiv \(\frac{a}{p}\)^{\frac{p+1}{2}}\pmod p.$$
\end{proof}

Before the statement of our second lemma, we first observe the following fact.
When $p\equiv 1\pmod 4$, we may write $p=a^2+b^2$ with $a,b\in\Z$ and $2\nmid a$. In this case, $\Q(\zeta_p)$ has a unique quartic subfield
$\Q(\delta_p)$, where $$\delta_p=\sqrt{\(\frac{2}{p}\)2p+2a\sqrt{p}}.$$
In fact, let
$$g(4)=\sum_{0\le t\le p-1}\zeta_p^{t^4}.$$
Then $g(4)$ has close connection with the quartic Gauss sum. Readers may see \cite{BE} for more details. And we have
(cf. \cite[(4.6)]{BE})
$$g(4)=\sqrt{p}\pm\sqrt{\(\frac{2}{p}\)2p+2a\sqrt{p}}.$$
Hence $\Q(\delta_p)$ is exactly the unique quartic subfield of $\Q(\zeta_p)$.

\begin{lemma}\label{Lem. Galois}
Let $p$ be an odd prime and let $D_p$ be a matrix of the form
$$D_p=(\zeta_p^{j^2k^2})_{0\le j,k\le m}.$$
We have the following results.

{\rm (i)} If $p\equiv 3\pmod 4$, then we may write
$$\det[D_p]=(u_p+v_pi\sqrt{p})$$
with $u_p,v_p\in\frac12 \Z$.

{\rm (ii)} If $p\equiv 1\pmod 4$ of the form $a^2+b^2$ with $a,b\in\Z$ and $2\nmid a$, then we may write
$$\det[D_p]=(\alpha_p+\beta_p\sqrt{p})\sqrt{\(\frac{2}{p}\)2p+2a\sqrt{p}}$$
with $\alpha_p,\beta_p\in\Q$.
\end{lemma}

\begin{proof}
Let
For each $a\in\Z$ with $\gcd(a,p)=1$, the automorphism $\sigma_a\in\Gal (\Q(\zeta_p)/\Q)$ is defined by sending $\zeta_p$ to $\zeta_p^a$.
Then we have
$$\sigma_{a^2}(\det[D_p])=\det[(\zeta_p^{a^2j^2k^2})]_{0\le j,k\le m}=\(\frac{a}{p}\)^{\frac{p+1}{2}}\det[D_p].$$
The last equation follows from Lemma \ref{Lem. permutation}. Hence by the Galois correspondence we obtain
$$\det[D_p]\in\begin{cases}\Q(i\sqrt{p})&\mbox{if}\ p\equiv 3\pmod4,
\\\Q(\delta_p)&\mbox{if}\ p\equiv 1\pmod4.\end{cases}$$
And when $p\equiv 1\pmod 4$, we also have
$$(\det[D_p])^2\in\Q(\sqrt{p}).$$

When $p\equiv3\pmod4$, clearly $\det[D_p]$ is an algebraic integer. Hence we may write
$$\det[D_p]=u_p+v_pi\sqrt{p}$$ with $u_p,v_p\in\frac12\Z$.

When $p\equiv 1\pmod 4$, we first write
$$\det[D_p]=x_p+y_p\delta_p$$ with $x_p,y_p\in\Q(\sqrt{p})$.
Since
$$(\det[D_p])^2=x_p^2+y_p^2\delta_p^2+2x_py_p\delta_p\in\Q(\sqrt{p}),$$
we have $x_py_p=0$. Noting that $\det[D_p]\not\in\Q(\sqrt{p})$, we therefore obtain $x_p=0$.
Writing $y_p=\alpha_p+\beta_p\sqrt{p}$, we get the desired result.
The proof is now complete.
\end{proof}

$\det[D_p]$ has close relations with $\det[C_p]$. In fact, it is easy to verify the identity
\begin{equation}\label{equation C}
\det[D_p]=(-1)^m\times\prod_{1\le k\le m}(1-\zeta_p^{k^2})\times \det[C_p].
\end{equation}
Now we are in the position to prove our main results.
\medskip

\noindent{\it Proof of Theorem}\ 1.1. Let $p>3$ be a prime with $p\equiv 3\pmod4$. We have (cf. \cite[p.71 Proposition 6.3]{IR})
$$\(\frac{j^2+k^2}{p}\)=\frac{1}{i\sqrt{p}}\(1+2\sum_{1\le t\le m}\zeta_p^{t^2(j^2+k^2)}\)$$
for any $1\le j,k\le m$. Let $\widetilde{D_p}$ be a matrix of the form $(a_{jk})_{0\le j,k\le m}$, where
$$a_{jk}=\begin{cases}1&\mbox{if}\ k=0,
\\2\zeta_p^{j^2k^2}&\mbox{otherwise}\ .\end{cases}$$
Then we have
$$\widetilde{D_p}D_p=i\sqrt{p}\cdot E_p,$$
where $E_p$ is a matrix of the form $(e_{jk})_{0\le j,k\le m}$ with
$$e_{jk}=\begin{cases}-i\sqrt{p}&\mbox{if}\ j=k=0,
\\\(\frac{j^2+k^2}{p}\)&\mbox{otherwise}\ .\end{cases}$$
For each $1\le k\le m$, we have (cf. \cite[p.63]{IR})
$$1+2\sum_{1\le j\le m}\(\frac{j^2+k^2}{p}\)=-1.$$
This shows that
$$\sum_{1\le j\le m}\(\frac{j^2+k^2}{p}\)=-1.$$
Hence for each $1\le k\le m$ the sum of entries in the $k$-th column of $E_p$ is equal to $0$. Thus it is easy to verify the following identity
\begin{equation}\label{equation D}
2^m(\det[D_p])^2=(-p)^{\frac{m+1}{2}}(-i\sqrt{p}+m)\cdot\det[S(p)]
\end{equation}
Using the notations in Lemma \ref{Lem. Galois}, we obtain the following equations
\begin{align}\label{equation E}
2^m(u_p^2-pv_p^2)&=(-p)^{\frac{m+1}{2}}\cdot m\cdot\det[S(p)],\\
2^{m+1}u_pv_p&=-1\cdot(-p)^{\frac{m+1}{2}}\cdot\det[S(p)].
\end{align}
Sun \cite{S19} proved that $p\nmid\det[S(p)]$.
This gives that
\begin{equation}\label{equation F}
\frac{u_p^2-pv_p^2}{u_pv_p}=\frac{u_p}{v_p}-p\frac{v_p}{u_p}=-2m.
\end{equation}
If $\nu_p(u_p)<\nu_p(v_p)$, then by (\ref{equation F}) $u_p/v_p$ is a $p$-adic integer. We get a contradiction.
Suppose now $\nu_p(u_p)\ge\nu_p(v_p)$. This implies that $pv_p/u_p$ is a $p$-adic integer. Hence we have
\begin{equation}\label{equation G}
1+\nu_p(v_p)\ge \nu_p(u_p)\ge \nu_p(v_p).
\end{equation}
As $p\nmid \det[S(p)]$, by (\ref{equation E}) we have $\nu_p(u_p)+\nu_p(v_p)=\frac{m+1}{2}$. Hence
$$\nu_p(u_p)=\begin{cases}\nu_p(v_p)+1=\frac{p+5}{8}&\mbox{if}\ p\equiv 3\pmod8,
\\\nu_p(v_p)=\frac{p+1}{8}&\mbox{if}\ p\equiv 7\pmod8.\end{cases}$$
By (\ref{equation C}) we know that
\begin{align*}
\det[D_p]&=(-1)^m\times\prod_{1\le k\le m}(1-\zeta_p^{k^2})\times \det[C_p]\\
&=-1\cdot(-1)^{\frac{h(-p)+1}{2}}i\sqrt{p}\cdot\det[C_p].
\end{align*}
Hence
$$\det[C_p]=(-1)^{\frac{h(-p)+1}{2}}(-v_p+\frac{u_p}{p}i\sqrt{p}).$$
Let $a_p=-v_p$ and let $b_p=u_p/p$. We get the desired result. This completes our proof.\qed
\medskip

\noindent{\it Proof of Theorem}\ 1.2. Let $p\equiv 1\pmod 4$ be a prime and let $\triangle$ be a quadratic non-residue modulo $p$.
Let $D_p^{\triangle}=(\zeta_p^{\triangle j^2k^2})_{0\le j,k\le m}$ and let $\widetilde{D_p}$ be as in the proof of Theorem 1.1. Then
it is easy to see that
\begin{equation}\label{equation H}
\widetilde{D_p}D_p^{\triangle}=\sqrt{p}\cdot F_p,
\end{equation}
where $F_p=(f_{jk})_{0\le j,k\le m}$ with
$$f_{jk}=\begin{cases}\sqrt{p}&\mbox{if}\ j=k=0,
\\\(\frac{j^2+\triangle k^2}{p}\)&\mbox{otherwise}\ .\end{cases}$$
Noting that
$$\det[F_p]=\sqrt{p}\cdot \det[S(\triangle,p)]+\det[T(\triangle,p)],$$
we therefore obtain that
\begin{equation}\label{equation I}
2^m\cdot \det[D_p]\cdot \sigma_{\triangle}(\det[D_p])=(\sqrt{p})^{m+1}\cdot (\sqrt{p}\cdot \det[S(\triangle,p)]+\det[T(\triangle,p)]),
\end{equation}
where $\sigma_{\triangle}$ is the automorphism in $\Gal(\Q(\zeta_p)/\Q)$ with $\sigma_{\triangle}(\zeta_p)=\zeta_p^{\triangle}$.
Let the notations be as in Lemma \ref{Lem. Galois}. Noting that
$$\sigma_{\triangle}\(\sqrt{\(\frac{2}{p}\)2p+2a\sqrt{p}}\)=\pm \sqrt{\(\frac{2}{p}\)2p-2a\sqrt{p}},$$
we have
\begin{align*}
\pm2^m(\alpha_p^2-p\beta_p^2)\sqrt{4p^2-4ap}&=\pm2^{m+1}b(\alpha_p^2-p\beta_p^2)\sqrt{p}\\
&=p^{\frac{m+2}{2}}\det[S(\triangle,p)]+p^{\frac{m}{2}}\det[T(\triangle,p)]\sqrt{p}.
\end{align*}
This shows that
\begin{align}\label{equation J}
\det[S(\triangle,p)]=&0,\\
\pm2^{m+1}b(\alpha_p^2-p\beta_p^2)=&p^{\frac{m}{2}}\det[T(\triangle,p)].
\end{align}
By (\ref{equation C}) our desired result follows from the equation
\begin{align*}
\det[D_p]=&\prod_{1\le k\le m}(1-\zeta_p^{k^2})\times \det[C_p]\\
=&\ve_p^{-h(p)}\sqrt{p}\times \det[C_p].
\end{align*}
This completes the proof.\qed

\maketitle

\Ack\ This research was supported by the National Natural Science Foundation of China (Grant No. 11571162).

\end{document}